\documentclass{interact}

\usepackage{anyfontsize}
\usepackage{mathtools}
\usepackage{hyperref}
\usepackage{color}
\usepackage{biblatex}
\renewcommand{\email}[1]{\href{mailto:#1}{\texttt{#1}}}
\newcommand{\orcid}[1]{\href{https://orcid.org/#1}{\texttt{#1}}}

\newtheorem{theorem}{Theorem}
\newtheorem*{theorem*}{Theorem}
\newtheorem{proposition}{Proposition}
\theoremstyle{definition}

\newtheorem{lemma}{Lemma}
\newtheorem{conj}{Conjecture}

\makeatletter
\def\thm@space@setup{%
  \thm@preskip=\parskip
  \thm@postskip=0pt
}
\newcommand{\monthyear}[1]{%
  \def\@monthyear{\uppercase{#1}}}
\newcommand{\volnumber}[1]{%
  \def\@volnumber{\uppercase{#1}}}
\makeatother

\title{On periods of the Tribonacci sequence modulo primes}
\author{
\name{
    Liam Baker\textsuperscript{a}\thanks{Corresponding author. Email address: \email{liambaker@sun.ac.za} \ \textbar\ ORCID: \orcid{0000-0003-2728-3963}} \and
    Florian Luca\textsuperscript{b}\thanks{Email address: \email{fluca@sun.ac.za} \ \textbar\ ORCID: \orcid{0000-0003-1321-4422}} \and
    Kerry Porrill\textsuperscript{c}\thanks{Email address: \email{kerry.porrill@gmail.com} \ \textbar\ ORCID: \orcid{0009-0002-1081-2128}}
}
\affil{
    \textsuperscript{a,b,c}Department of Mathematical Sciences,
    Stellenbosch University,
    Stellenbosch, South Africa \\
    \textsuperscript{b}Max-Planck Institute for Software Systems, Saarbr\"ucken, Germany \\
    \textsuperscript{c}Mathematical Institute,
    University of Oxford,
    Oxford, United Kingdom
}
}
\begin{document}

\monthyear{Month Year}
\volnumber{Volume, Number}
\setcounter{page}{1}

\maketitle

\textbf{Article type}: research

\bigskip
\begin{abstract}
    In this paper, we study the $p$-adic zeros of the Tribonacci sequence $T_n$, for a prime $p$.
    In particular, we complete the analysis that was done in \cite{MathComp} by investigating the case of $p=11$.
    % In addition, we investigate the $p$ for which $T_n$ and $T_{n-1}$ are both $0$ modulo $p$ and one or more of them are divisible by $p^2$, proving under the assumption of a generalisation of the $abc$ conjecture that there are infinitely many such $p$.
    In addition, we investigate the primes $p$ for which at the first two values $T_n$ and $T_{n-1}$ which are both divisible by $p$, both of these are actually divisible by $p^2$, proving under the assumption of a generalisation of the $abc$ conjecture that there are infinitely many $p$ which are \emph{not} such.
\end{abstract}



\section{Introduction}

Let $(T_n)_{n\in \mathbb{Z}}$ be the Tribonacci sequence given by $T_0=0$, $T_1=T_2=1$ and 
$$
T_{n+3}=T_{n+2}+T_{n+1}+T_n\qquad \text{for all}\quad n\in \mathbb{Z}.
$$
This recursive sequence has characteristic polynomial $P(X) = X^3-X^2-X-1$, with distinct roots $\Lambda = \{\lambda_1,\lambda_2,\lambda_3\} \subset \mathbb{C}$.
From this we have the Binet formula
\begin{equation} \label{eq:binet}
    T_n = \sum_{\lambda\in \Lambda} c_{\lambda}\lambda^{n}\quad\text{for}\ n\in \mathbb{Z},
\end{equation}
where each $c_\lambda=(\lambda P'(\lambda))^{-1}$.

For a prime number $p$ and a nonzero integer $m$ let $\nu_p(m)$ be the exponent of $p$ in the factorization of $m$, with $\nu_p(0) \coloneq \infty$.
In \cite{ML}, Marques and Lengyel determined $\nu_2(T_n)$:
\begin{theorem*}
For $n\ge 1$, we have 
$$
\nu_2(T_n) =
% \left\{ \begin{matrix} 0, & \text{if} & n\equiv 1,2 \pmod 4;\\
% 1, & \text{if} & n\equiv 3,11\pmod {16};\\
% 2, & \text{if} & n\equiv 4,8\pmod {16};\\
% 3, & \text{if} & n\equiv 7\pmod {16};\\
% \nu_2(n)-1, & \text{if} & n\equiv 0\pmod {16};\\
% \nu_2(n+4)-1, & \text{if} & n\equiv 12\pmod {16};\\
% \nu_2(n+17)+1, & \text{if} & n\equiv 15\pmod {32};\\
% \nu_2(n+1)+1, & \text{if} & n\equiv 31\pmod {32}.
% \end{matrix}
% \right.
\begin{cases}
    1               & n\equiv 3,11\pmod {16} \\
    2               & n\equiv 4,8\pmod {16} \\
    3               & n\equiv 7\pmod {16} \\
    \nu_2(n)-1      & n\equiv 0\pmod {16} \\
    \nu_2(n+4)-1    & n\equiv 12\pmod {16} \\
    \nu_2(n+17)+1   & n\equiv 15\pmod {32} \\
    \nu_2(n+1)+1    & n\equiv 31\pmod {32}
\end{cases}
$$
\end{theorem*}

In fact, they stated their theorem only for $n\ge 1$, but it is easy to see that it holds for all $n\in \mathbb{Z}$.

Encouraged by their result, they set forward a conjecture claiming that similar congruences should hold for other primes:

% \liam{I'm unsure about making this the same theorem layout as Theorem 1. Thoughts?}
\begin{conj}[Marques, Lengyel]
Let $p$ be a prime number.
There exists an integer $Q$ such that for each $i\in \{0,1,\ldots,Q-1\}$ one of the following holds:
\begin{itemize}\item[(C)] There exists $\kappa_i\in \mathbb{Z}_{\ge 0}$ such that for integer $n\equiv i\pmod Q$ we have $\nu_p(T_{n}) = \kappa_i$.
\item[(L)] There exist integers $a_i,\kappa_i,\nu_i\ge 1$ with $\nu_p(a_i-i)\ge \nu_p(Q)$ such that if $n\in \mathbb{Z}$ with $n\equiv i\pmod Q$ we have
$$
\nu_p(T_n)=\kappa_i+\nu_i \nu_p(n-a_i).
$$
\end{itemize}
\end{conj}

In fact, Marques and Lengyel allowed for the above conjecture to hold for all but finitely many positive integers $n$.
In \cite{MathComp} it was noted that it would be better to allow the $a_i$ to be rational numbers in item (L) above.
The main result in \cite{MathComp} is an algorithm which decides, given a prime $p \neq 11$, whether the above conjectures hold or fail.
An implementation of the algorithm with primes $p < 10^4$ found a few primes for which the conjectures hold and many primes for which they fail, and in some cases the method failed to make a decision. 
The prime $11$ was excluded because the characteristic polynomial $X^3-X^2-X-1$ has a double root modulo $11$, but the analysis is now completed in Section \ref{sec:p=11}.

In Section \ref{sec:consec_zero}, we investigate the phenomenon of the occurrence of consecutive zeroes of the Tribonacci sequence modulo $p$ (which is related to the period of $T_n$ modulo $p$) and in particular the highest powers of $p$ dividing these consecutive terms, since that affects the powers of $p$ dividing subsequent consecutive zeroes modulo $p$, culminating in both and algebraic and an elementary proof of the following result:
\begin{theorem*}
Let $p\neq2$ be a prime.
Then for $N_0 \coloneq N_0(p)$,
\begin{enumerate}
    \item If $\nu_p(T_{N_0})\neq\nu_p(T_{N_0-1})$, then there is no $n \in \mathbb{Z}$ such that $\nu_p(T_n) = \nu_p(T_{n-1})>0$.
    \item If $\nu_p(T_{N_0})=\nu_p(T_{N_0-1})$, then we have that when $n$ is a multiple of $N_0$, the following holds: $\nu_p(T_n) = \nu_p(T_{n-1}) = \nu_p\left(\frac{n}{N_0}\right)+\nu_p(T_{N_0})$.
\end{enumerate}
\end{theorem*}

Finally, in Section \ref{sec:weifrich} we define \emph{Tribonacci-Weifrich} primes as primes $p$ for which the first two consecutive terms of $T_n$ which are divisible by $p$ are in fact both divisible by $p^2$, and use some recent work of Fellini and Murty \cite{FM} to show that, under the $abc$ conjecture for number fields, there are infinitely many primes which are \emph{not} Tribonacci-Weifrich.

\section{The case \texorpdfstring{$p=11$}{p=11}} \label{sec:p=11}

\subsection{The \texorpdfstring{$11$}{11}-adic roots of the Tribonacci polynomial}

Let $P(X)=X^3-X^2-X-1$.
We have
$$
P(X)\equiv (X-7)^2(X-9)\pmod {11}.
$$
The root $9\pmod {11}$ is simple and $P'(\alpha)=3\alpha^2-2\alpha-1\equiv 4\pmod {11}$.
Thus by Hensel's lemma (Proposition 3.3 in \cite{MathComp}), there exists a unique $p$-adic integer 
$\alpha$ such that $P(\alpha)=0$ in $\mathbb{Q}_{11}$ and $\nu_{11}(\alpha-9)\ge 1$.
We can even write down the beginning of its expansion:
\begin{align*}
    \alpha & = \begin{multlined}[t]
        9+2\cdot 11+1\cdot 11^2+2\cdot 11^3+2\cdot 11^4+2\cdot 11^5+10\cdot 11^6\\
        + 0\cdot 11^7+1\cdot 11^8+4\cdot 11^9+4\cdot 11^{10}+\dotsb
    \end{multlined}\\
    & = {\overline{921222(10)0144\cdots}}
\end{align*}
It seems that we are in trouble about the double root congruent to $7\pmod {11}$.
However, this is not actually a double root, but $P(X)$ has two quadratic roots $11$-adically which modulo $11$ coincide to $7\pmod {11}$.
To see why, note that since $\alpha$ is a root of $P(X)$ we have that
$$
P(X)=(X-\alpha)(X^2-(1-\alpha)X+1/\alpha).
$$
Note that $Q(X)=X^2-(1-\alpha)X+1/\alpha$ is a quadratic polynomial in $\mathbb{Z}_{11}[X]$.
It has two roots 
$$
\lambda, {\overline {\lambda}} =\frac{(1-\alpha)\pm {\sqrt{(1-\alpha)^2-4/\alpha}}}{2}.
$$
If we take the approximation $\alpha_1=9+2\cdot 11\pmod {11^2}\equiv 31\pmod {11^2}$, then 
$$
(1-\alpha_1)^2-4/\alpha_1\equiv 88\pmod {11^2},
$$
so an approximation of $\lambda,~{\overline{\lambda}}$ in $\mathbb{Q}_{11}[{\sqrt{22}}]$ is 
$$
\lambda_1,~{\overline{\lambda}}_1=\frac{(1-9)\pm {\sqrt{88}}}{2}\pmod {11}=7\pm {\sqrt{22}}\pmod {11}.
$$
To see whether this is a good enough approximation, we calculate 
$$
P(\lambda_1)=2\cdot 11^{3/2}(3{\sqrt{11}}+7{\sqrt{2}}),\qquad P'(\lambda_1)=2\cdot 11^{1/2}(9{\sqrt{11}}+20{\sqrt{2}}).
$$
Extending $\nu_{11}$ to $\mathbb{K}\coloneq \mathbb{Q}[{\sqrt{22}}]$ in the obvious way and putting $|x|_{11}=11^{-\nu_{11}(x)}$ for $x\in \mathbb{Q}_{11}[{\sqrt{22}}]$, we get that
$$
|P(\lambda_1)|_{11}=11^{-3/2} \qquad \text{and} \qquad |P'(\lambda_1)|_{11}=11^{-1/2}.
$$
In particular,
$$
|P(\lambda_1)|_{11}<|P'(\lambda_1)|_{11}^{2}.
$$
This shows via a more general version of the Hensel's lemma as in \cite[Proposition 3.3]{MathComp}, that  there is a unique root $\lambda$ of $P(x)$ in $\mathbb{Q}_{11}[{\sqrt{22}}]$ such that 
$$
|\lambda-\lambda_1|_{11} < |P'(\lambda_1)|_{11} = 11^{-1/2}.
$$
By the same argument, there is a unique root ${\overline{\lambda}}$ of $P(X)$ in $\mathbb{Q}_{11}[{\sqrt{22}}]$ such that 
$$
|{\overline{\lambda}}-{\overline{\lambda}}_1|_{11}<|P'(\lambda_1)|_{11}=11^{-1/2}.
$$
Note that $\lambda\ne {\overline{\lambda}}$, for if they were equal then the triangle inequality would give us 
$$
11^{-1/2}=|2{\sqrt{22}}]_{11}=|\lambda_1-{\overline{\lambda_1}}|_{11}\le \max\{|\lambda-\lambda_1|_{11},|\lambda-{\overline{\lambda_1}}|_{11}\}<11^{-1/2},
$$
a contradiction.
In fact, we have 
\begin{align*}
    \lambda_2 & = (2+9\cdot 11)+{\sqrt{22}}, & P(\lambda_2) & = 11^2(9807+277{\sqrt{22}}); \\
    \lambda_3 & =  (2+9\cdot 11)+{\sqrt{22}} (1+9\cdot 11), & P(\lambda_3) & = 11^{5/2}(53283{\sqrt{11}}+209500{\sqrt{2}}); \\
    \lambda_4 & =  (2+9\cdot 11+5\cdot 11^2)+{\sqrt{22}}(1+9\cdot 11), & P(\lambda_4)  & = 11^3(446439+94900{\sqrt{22}}).
\end{align*}

\subsection{The \texorpdfstring{$11$}{11}-adic zeros of the Tribonacci sequence}

The period of the Tribonacci sequence modulo $11$ is $110$.
This can be easily checked computationally or from the above formulas.
Namely, it is the minimum positive integer $N$ such that 
$$
\lambda^N\equiv {\overline{\lambda}}^N\equiv \alpha^N\equiv 1\pmod {11}.
$$
Using $\alpha_1\equiv 9\pmod {11}$, we get $9^N\equiv 1\pmod {11}$, so $5\mid N$.
Using $\lambda_1$, we get that 
$$
(7+{\sqrt{22}})^N\equiv 1\pmod {11}.
$$
Expanding the left--hand side with the Newton binomial formula we get
$$
7^N+  7^{N-1}N{\sqrt{22}} \equiv 1\pmod {11}.
$$
This gives that $11\mid 7^N-1$, so $10\mid N$ since $7$ is a primitive root modulo $11$.
Next, we are left with 
$$
 7^{N-1}N{\sqrt{22}}\equiv 0\pmod {11},
$$
which gives $11\mid N$.
Hence, $N=110$.
We now do a quick search in the interval $n\in [0,109]$ to find out the values of $n$ in this interval such that $T_n\equiv 0\pmod {11}$ and we get
$$
\mathcal{Z}_{11} = \{{\begin{color}{blue} 0\end{color}}, 8, 14, {\begin{color}{red}35\end{color}}, {\begin{color}{red}37\end{color}}, 42, 61, {\begin{color}{blue}93\end{color}}, {\begin{color}{blue}106\end{color}}, {\begin{color}{blue}109\end{color}}\}.
$$
We understand the blue ones.
They are congruent to the integer zeros $\mathcal{Z}_\mathbb{Z}(T)=\{-17,-4,-1,0\}$ modulo $N=110$.
We also understand the red ones.
It was shown in \cite{MathComp} that for certain determinations of the cubic roots $\lambda^{1/3}$ we have that 
$$
\sum_{\lambda\in\Lambda} c_{\lambda} \lambda^{1/3}=0,
$$
so we can say that $1/3$ is a \emph{rational zero} of the Tribonacci sequence.
The same holds for $-5/3$ and the set $\mathcal{Z}_\mathbb{Q}(T)=\{-17,-4,-5/3,-1,0,1/3\}$ consists of all the rational zeros of the Tribonacci sequence.
The rational zeros $1/3$ and $-5/3$ are present modulo $p$ whenever the period $N_p$ is coprime to $3$, since then $-5/3$ and $1/3$ make sense modulo $N_p$.
For us, $p=11$, $N_{11}=N=110$ is coprime to $3$ so $1/3\pmod {110}$ and $-5/3\pmod {110}$ make sense and they end up being $37\pmod {110}$ and $35\pmod {110}$, respectively. 

To lift the zeros in $\mathcal{Z}_{11}$ modulo $11$ to the $11$-adic setting, we proceed as in \cite{MathComp}.
For each $\ell\in \mathcal{Z}_{11}$, we look for a zero of the form 
$\ell+110z$, with some $p$-adic number $z$.
Then 
$$
f_{\ell}(z)=c_{\alpha} \alpha^{\ell} \exp\left(z\log \alpha^{110}\right)+c_{\lambda} \lambda^{\ell} \exp\left(z\log \lambda^{110}\right)+c_{\overline{\lambda}} {\overline{\lambda}}^\ell 
\exp\left(z\log {\overline{\lambda}}^{110}\right).
$$
The condition to be able to lift is that if we denote by $g(z) \coloneq f_{\ell}(z)/p$, then we need $g'(0)\not\equiv 0\pmod p$.
This is equivalent (see (8.2) in \cite{MathComp}) to 
$$
T(\ell+N)\not\equiv T(\ell)\pmod {p^2}.
$$
We checked that this is the case for $p=11,~N=110$ and all  the ten values of $\ell$ in $\mathcal{Z}_{11}$.
So, in conclusion $(T_n)_{n\in \mathbb{Z}}$ has $10$ $11$-adic zeros, $6$ of which arise from the integer and rational zeros of the Tribonacci sequence and four which do not.
The ones which do not start with $a\pmod {110}$ where $a \in \{8,14,42,61\}$.
The remaining ones satisfy a Marques-Lengyel condition namely
$$
\nu_{11}(T_n) = 
\begin{cases}
    \nu_p(n+c)+1    & n \pmod {110} \in \mathcal{Z}_\mathbb{Q}(T);\\
    0               & n \pmod {110} \not\in \mathcal{Z}_\mathbb{Q}(T)\cup \{8,14,42,61\}.
\end{cases}
$$
% \liam{is the second case correct? And what is $A$?}

\section{Consecutive \texorpdfstring{$p$}{p}-adic zeroes of \texorpdfstring{$T_n$}{the Tribonacci sequence}} \label{sec:consec_zero}

Since $T_{-1}=T_0=0$, for every integer $m$ there exists a smallest positive integer $N_0(m)$ such that $T_{N_0}\equiv T_{N_0-1}\equiv 0\pmod m$.

\begin{theorem}\label{mainr}
Let $p\neq2$ be a prime.
Then for $N_0 \coloneq N_0(p)$,
\begin{enumerate}
    \item \label{p1} If $\nu_p(T_{N_0})\neq\nu_p(T_{N_0-1})$, then there is no integer $n$ such that $\nu_p(T_n)=\nu_p(T_{n-1})>0$.
    \item \label{p2} If $\nu_p(T_{N_0})=\nu_p(T_{N_0-1})$, then we have that when $n$ is a multiple of $N_0$, the following holds: $\nu_p(T_n) = \nu_p(T_{n-1}) = \nu_p\left(\frac{n}{N_0}\right)+\nu_p(T_{N_0})$.
\end{enumerate}
\end{theorem}

We include two proofs of this result: first an elementary proof in Subsection \ref{ssec:elementary}, then an algebraic proof in Subsection \ref{ssec:algebraic}.

\subsection{Elementary Proof} \label{ssec:elementary}
Firstly some lemmas, the first two of which are well known so we refrain from including proofs.
\begin{lemma}\label{lem1}
    There is no prime that divides three consecutive terms of the Tribonacci sequence.
\end{lemma}
% \begin{proof}
%     Notice that if $p\mid T_n,T_{n+1},T_{n+2} $ then we have that $p\mid T_{n+2}-T_{n+1}-T_n$, thus $p\mid T_{n-1}$, inductively we get that $p\mid T_1$, which is a contradiction. 
% \end{proof}

\begin{lemma}[Agromonof's Identity] \label{lem2}
    For integers $n$ and $k$:
    $$ T_{n+k}=T_{k-1}T_{n-1}+(T_{k+1}-T_k)T_n+T_kT_{n+1}. $$
\end{lemma}
% \begin{proof}
%     Let $n$ be fixed. Clearly if we sum the equations for $3$ consecutive values of $k$, we get the equation for the next value of $k$. Thus if we can show what the equation holds for $3$ consecutive values of $k$, then it must hold for all values of $k$. Since $T_{-1}=T_0=0$, and $T_{-2}=T_1=T_2=1$, it is easy to verify that the recurrence holds for $k=-1,0,1$.
% \end{proof}

%\liam{in this and the following lemmas, what exactly are $p$ and $k$?}
\begin{lemma}\label{lem3}
    Let $p$ be a prime and let $k$ and $n$ be positive integers.
    Then $p^k\mid T_n$ and $p^k\mid T_{n-1}$ both hold if and only if $N_0(p^k) \mid n$. 
\end{lemma}
\begin{proof}
    Firstly let $T_{N_0(p^k)+1} = a \pmod{p^k}$; clearly by Lemma \ref{lem1}, $a$ is not divisible by $p$.
    Note that for $n > N_0(p^k)$ we have, by Lemma \ref{lem2}, that $T_{n}=aT_{n-N_0(p^k)}\pmod{p^k}$; inductively it follows that for any positive integer $n$, $p^k\mid T_n$ and $p^k\mid T_{n-1}$ both hold if and only if $N_0(p^k)\mid n$.
\end{proof}

\begin{lemma}\label{lem4}
Let $p$ be a prime and let $k$ be a positive integer.
If $p^k\mid T_n$ and $p^k\mid T_{n-1}$ for some positive integers $n$ and $k$, then the following congruences hold for every positive integer $m$:
\begin{align*}
    T_{mn-1} &\equiv T_{n+1}^{m-2}\left[m(T_{n+1}T_{n-1})+\tfrac{m(m-1)}{2}(T_n^2-2T_nT_{n-1}-T_{n-1}^2)\right] \pmod{p^{k+2}},\\
    T_{mn} &\equiv  T_{n+1}^{m-2}\left[m(T_{n+1}T_{n})+\tfrac{m(m-1)}{2}(-T_n^2+T_{n-1}^2)\right] \pmod{p^{k+2}},\\
    T_{mn+1} &\equiv T_{n+1}^{m-2}\left[T_{n+1}^2+\tfrac{m(m-1)}{2}(T_n^2+2T_nT_{n-1})\right] \pmod{p^{k+2}}.
\end{align*}
\end{lemma}
% \liam{Do we want to include negative $m$? We would have to extend the proof.}
\begin{proof}
Let $n$ be fixed.
We will prove this lemma by induction on $m$.

\textbf{Base Case:} $m=1$, clearly $p\nmid T_{n+1}$, thus $T_{n+1}^{-1}$ is well-defined $\!\pmod {p^{k+2}}$.
With this in mind, it is easy to verify that the congruences hold for $m=1$.

\textbf{Inductive step:}
Assume the congruences hold for some positive integer $m$.
We will prove that they hold for $m+1$.
By Lemma \ref{lem2} we have that
\begin{align*}
    T_{(m+1)n-1} &= T_{(n-1)+mn} =(T_{n+1}-T_n-T_{n-1})T_{mn-1}+(T_{n}-T_{n-1})T_{mn}+T_{n-1}T_{mn+1}.
\end{align*}
From here we can invoke the induction hypothesis noting that any term $T_n^{a}T_{n-1}^b$ where $a+b=3$ will be divisible by $p^{k+2}$, so with simple algebra we get the desired congruence.
Similarly, using
\begin{align*}
    T_{(m+1)n} &= T_{n+mn} = T_{n-1}T_{mn-1}+(T_{n+1}-T_{n})T_{mn}+T_{n}T_{mn+1}
    \shortintertext{and}
    T_{(m+1)n+1} &= T_{(n+1)+mn} = T_{n}T_{mn-1}+(T_{n+2}-T_{n+1})T_{mn}+T_{n+1}T_{mn+1}
\end{align*}
we get the other congruences.
The proof that it holds for $m-1$ if it holds for $m$ is similar; thus we also get that the congruences hold for negative $m$.
\end{proof}

\begin{lemma}\label{lem5}
    Let $p\neq2$ be a prime.
    Suppose that $i = \nu_p\left({T_n}\right)$ and $j = \nu_p\left(T_{n-1}\right)$ are both positive and let $k=\min(i,j)$.
    Then the following hold for every integer $m$:
    \begin{enumerate}
        \item $\nu_p\left({T_{mn-1}}\right)\geqslant  k+\nu_p(m)$, with equality if and only if $j\leqslant i$.
        \item $\nu_p\left({T_{mn}}\right)\geqslant  k+\nu_p(m)$, with equality if and only if $i\leqslant j$.
    \end{enumerate}
\end{lemma}
\begin{proof}
We will show this by induction on $\nu_p(m)$.
Firstly, suppose that $m$ is not divisible by $p$.
Note that the first congruence in Lemma \ref{lem4} also holds modulo $p^{k+1}$; thus we have that
% \liam{Is this correct?}
\[ T_{mn-1}\equiv mT_{n+1}^{m-2}(T_{n+1}T_{n-1}) \pmod{p^{k+1}} \]
since the second term is clearly divisible by $p^{k+1}$.
From this we can clearly see that $p^{k}\mid T_{mn-1}$ and that $p^{k+1}\mid T_{mn-1}$ if and only if $p^{k+1}\mid T_{n-1}$, which is exactly the case when $j>i$.
Thus $\nu_p\left({T_{n-1}}\right)\geqslant  k$, with equality if and only if $j\leqslant i$.
A similar argument shows that $\nu_p\left({T_{mn}}\right)\geqslant  k$, with equality if and only if $i\leqslant j$.
Thus the lemma is true when $p \nmid m$.

Suppose the lemma is true for all $m$ where $\nu_p(m)<t$ for some $t\geqslant 1$, and consider some $m$ where $\nu_p(m)=t$.
Let $m^\prime=m/p$, $k^\prime=k+\nu_p(m)$, and $n^\prime=nm^\prime$.
Since $\nu_p(m^\prime)<t$, we have that $\nu_p\left({T_{m^\prime-1}}\right)\geqslant  k^\prime$, $j\leqslant i$, and $\nu_p\left({T_{m^\prime n}}\right)\geqslant  k^\prime$, with equality if and only if $i\leqslant j$.
Notice that since $p$ is odd, we have that $p\mid \tfrac{p(p-1)}{2}$ and thus by Lemma \ref{lem4} we have that
\begin{align*}
    T_{mn-1} = T_{p(m^\prime n)-1} &\equiv pT_{m^\prime n+1}^{p-2}(T_{m^\prime n+1}T_{m^\prime n-1}) \\
    \text{and} \qquad
    T_{mn} = T_{p(m^\prime n)} &\equiv pT_{m^\prime n+1}^{p-2}(T_{m^\prime n+1}T_{m^\prime n}) \pmod{p^{k^\prime +2}}.
\end{align*}
The lemma is therefore true for $m$.
\end{proof}

Now we prove Theorem \ref{mainr}.
\begin{proof}
Suppose for some prime $p$ we have that $\nu_p(T_{N_0})\neq\nu_p(T_{N_0-1})$.
By Lemma \ref{lem3}, we have that if both $\nu_p(T_n),\nu_p(T_{n-1})$ are non-zero then $N_0\mid n$.
Since $\nu_p(T_{N_0})\neq\nu_p(T_{N_0-1})$ we have, by Lemma \ref{lem5} that one of $\nu_p(T_{mN_0}-1),\nu_p(T_{mN_0})$ is equal to $\min(\nu_p(T_{N_0}),\nu_p(T_{N_0-1}))+\nu_p(m)$ and the other is strictly greater.
Hence, Item $\ref{p1}$ is true.

Now suppose that for some prime $p$ we have that $\nu_p(T_{N_0})=\nu_p(T_{N_0-1})$.
Thus by Lemma \ref{lem5} we have that if $n$ is a multiple of $N_0$ then $\nu_p(T_n)=\nu_p(T_{n-1})=\nu_p\left(\frac{n}{N_0}\right)+\nu_p(T_{N_0})$.
Hence, Item \ref{p2} holds as well.
\end{proof}

\subsection{Algebraic proof} \label{ssec:algebraic}
% Let $\{T_n\}_{n\ge 0}$ be the Tribonacci sequence.
For an integer $m$ let $N(m)$ be the period of the Tribonacci sequence modulo $m$.
It is well known that $N(2^k)=2^{k+1}$ for all $k\ge 1$.
Further, $N(11^k) = 10\times 11^k$ for all $k\ge 1$ and if $p\ne 2,11$, then $N(p^k)\mid p^{k-1}N(p)$, where 
% $$
% N(p)\mid \left\{\begin{matrix} p-1 & \text{if} & X^3-X^2-X-1 & {\text{\rm has~three~roots}}  & \pmod p\\
% p^2-1 & \text{if} & X^3-X^2-X-1 & {\text{\rm has  one ~root}} & \pmod p\\
% p^2+p+1 & \text{if} & X^3-X^2-X-1 & {\text{\rm has~no~roots}} & \pmod p
% \end{matrix}\right\}.
% $$
\[
N(p) \ \left| \ \begin{cases}
    p-1     & X^3-X^2-X-1 \ {\text{has three roots}}   \pmod p\\
    p^2-1   & X^3-X^2-X-1 \ {\text{has one root}}  \pmod p\\
    p^2+p+1 & X^3-X^2-X-1 \ {\text{has no roots}}  \pmod p
\end{cases} \right..
\]
% Recall that since $T_{-1}=T_0=0$, for every prime $p$ there exists a smallest positive integer $N_0\coloneq N_0(p)$ such that $T_{N_0}\equiv T_{N_0-1}\equiv 0\pmod p$.
Let $\lambda_p\coloneq T_{N_0+1}\pmod p$.
Thus, 
\begin{align*}
    (T_{N_0},T_{N_0+1},T_{N_0+2}) & \equiv (0,\lambda_p,\lambda_p) \equiv \lambda_p(0,1,1) \equiv \lambda_p(T_0,T_1,T_2) \pmod p.
\end{align*} 
If $k\coloneq k_p$ is the minimal positive integer such that $\lambda_p^k\equiv 1\pmod p$, we get that $N(p)=N_0(p)\cdot k$. 

We prove:
\begin{lemma} \label{lem:algebra}
    Assume $p\ne 2,11$.
    Then the following hold. 
    \begin{enumerate}
        \item[(i)] $N_0$ is the minimal positive integer such that $\lambda_1^{N_0}\equiv \lambda_2^{N_0}\equiv \lambda_3^{N_0}\pmod p$. 
        \item[(ii)] $k\in \{1,3\}$.
        Further, if $k=3$, then $p\equiv 1\pmod 3$.
    \end{enumerate}
\end{lemma}
% \begin{theorem}
%     Assume $p\ne 2,11$.
%     Then the following hold. 
%     \begin{itemize}
%         \item[(i)] $N_0$ is the minimal positive integer such that $\lambda_1^{N_0}\equiv \lambda_2^{N_0}\equiv \lambda_3^{N_0}\pmod p$. 
%         \item[(ii)] $k\in \{1,3\}$.
%         Further, if $k=3$, then $p\equiv 1\pmod 3$.
%         \item[(iii)] If $\nu_p(T_{N_0})\ne \nu_{p}(T_{N_0-1})$, then $\nu_p(T_n)\ne \nu_p(T_{n-1})$ for any multiple $n$ of $N_0$.
%         \item[(iv)] If $\nu_p(T_{N_0})=\nu_p(T_{N_0-1})$ then for any multiple $n$ of $N_0$ we have that 
%         $$
%         \nu_p(T_{n})=\nu_p(T_{n-1})=\nu_p(T_{N_0})+\nu_p(n/N_0).
%         $$
%     \end{itemize}
% \end{theorem}
% \kerry{perhaps we make (i) and (ii) into a lemma so that (iii) and (iv) agree with the main theorem?}
% \liam{I agree}

\begin{proof}
Recall that, by Equation \eqref{eq:binet}, 
$$
T_n=\sum_{i=1}^3 c_i \lambda_i^n,
$$
where $c_i=\lambda_i/P'(\lambda_i)$ for $i=1,2,3$, where $P(X)\coloneq X^3-X^2-X-1$.
We write the equations $T_{N_0}\equiv T_{N_0-1}\equiv 0\pmod p$ as 
$$
\sum_{i=1}^3 c_i \lambda_i^{N_0-j}=p^{a_j}b_j,\quad~\text{where}\quad a_j\ge 1\quad \text{and}\quad p\nmid b_j\quad \text{for}\quad j=0,1.
$$
This we rewrite as 
\begin{equation}
\label{eq:1}
d_{1,j}X_1+d_{2,j}X_2
% = \sum_{i=1}^2 d_{i,j} X_i
= 1+p^{a_j} e_j,\quad \text{for}\quad j=1,2,
\end{equation}
where
\begin{align*}
    X_i &= (\lambda_i/\lambda_3)^{N_0}, \\
    \left(\begin{matrix} d_{1,0} & d_{2,0}\\ d_{1,1} & d_{2,1}\end{matrix}\right) &\coloneq \left(\begin{matrix} c_1/c_3 & c_2/c_3 \\ (c_1/c_3)(\lambda_1/\lambda_3)^{-1} & (c_2/c_3)(\lambda_2\lambda_3)^{-1}\end{matrix}\right),
    \shortintertext{and}
    e_j &\coloneq b_j/(c_3\lambda_3^{N_0})\quad \text{for}\quad  j=0,1.
\end{align*}
We treat \eqref{eq:1} as a system of two equations in two unknowns $(X_1,X_2)$.
Solving it with Cramer's rule, we get 
$$
(X_1,X_2)=\left(\frac{\Delta_1}{\Delta},\frac{\Delta_2}{\Delta}\right),
$$
where 
$$
\Delta\coloneq \left|\begin{matrix} d_{1,0} & d_{1,1}\\ d_{2,0} & d_{2,1}\end{matrix}\right|=\frac{c_1c_2\lambda_3(\lambda_1-\lambda_2)}{c_3^2\lambda_1\lambda_2}.
$$
Since $c_i/c_j$ are conjugates for all $1\le i\ne j\le 3$ of minimal polynomial
$$
11X^6 + 33X^5 + 64X^4-73X^3 + 64X^2 + 33X + 11
$$
and $\lambda_i-\lambda_j$ is an algebraic integer dividing $\text{Res} (P(X),P'(X))=44$, it follows that $\Delta$ is an algebraic number which is a product of prime ideals (at integer exponents) dividing 
$2\times 11$.
Further, 
\begin{align*}
\Delta_1 & = \left|\begin{matrix} 1+p^{a_0}e_0  & (c_2/c_3)\\ 1+p^{a_1}e_1 & (c_2/c_3) (\lambda_2/\lambda_3)^{-1}\end{matrix} \right| = \left|\begin{matrix} 1 & (c_2/c_3)\\ 1 & (c_2/c_3)(\lambda_2/\lambda_3)^{-1}\end{matrix}\right| + \left|\begin{matrix} p^{a_0} e_0 & (c_2/c_3)\\ p^{a_1} e_1 & (c_2/c_3)(\lambda_2/\lambda_3)^{-1}\end{matrix}\right|\\
&=\Delta+(c_2/c_3) \left|\begin{matrix} p^{a_0}e_0 & 1\\ p^{a_1} e_1 & \lambda_3/\lambda_2\end{matrix}\right|.   
\end{align*}
Similarly, 
\begin{align*}
\Delta_2 & = \left|\begin{matrix} (c_1/c_3)  & 1+p^{a_0} e_0\\ (c_1/c_3)(\lambda_1/\lambda_3)^{-1}  & 1+p^{a_1}e_1\end{matrix} \right| = \left|\begin{matrix} (c_1/c_3) & 1\\ (c_1/c_3)(\lambda_1/\lambda_3)^{-1} & 1\end{matrix}\right| + \left|\begin{matrix} (c_1/c_3) & p^{a_0}e_0 \\ (c_1c_3)(\lambda_1/\lambda_3)^{-1} & p^{a_1}e_1\end{matrix}\right| \\
&= \Delta+(c_1/c_3) \left|\begin{matrix} 1 & p^{a_0}e_0 \\ \lambda_3/\lambda_1 & p^{a_1}e_1\end{matrix}\right|.
\end{align*}
Since $e_j=b_j/(c_3\lambda_3)^{N_0}$ for $j=0,1$, and $\lambda_3$ is a unit, it follows that 
\begin{align*}
X_1 & = 1+\eta_1 \left|\begin{matrix} p^{a_0} (b_0/\lambda_3^{N_0}) & 1\\ p^{a_1}( b_1/\lambda_3^{N_0} ) & \lambda_3/\lambda_2\end{matrix} \right| = 1+\eta_1 \gamma_1 \quad\text{and} \\
X_2 & = 1+\eta_2  \left|\begin{matrix} 1 & p^{a_0} (b_0/\lambda_3^{N_0} )\\ \lambda_3/\lambda_1 & p^{a_1} (b_1/\lambda_3^{N_0} )\end{matrix} \right| = 1+\eta_2\gamma_2.
\end{align*}
where 
$$
\eta_1\coloneq \frac{c_2/c_3}{\Delta}, \quad \eta_2\coloneq \frac{c_1/c_3}{\Delta},
$$
are algebraic numbers which are products of prime ideals to positive or negative powers dividing $2\times 11$, while 
$$
\gamma_1\coloneq \left|\begin{matrix} p^{a_0} (b_0/\lambda_3^{N_0}) & 1\\ p^{a_1}( b_1/\lambda_3^{N_0} ) & \lambda_3/\lambda_2\end{matrix} \right| \quad\text{and}\quad \gamma_2\coloneq  \left|\begin{matrix} 1 & p^{a_0} (b_0/\lambda_3^{N_0} )\\ \lambda_3/\lambda_1 & p^{a_1} (b_1/\lambda_3^{N_0} )\end{matrix} \right|
$$
are algebraic integers which are multiples of $p$.
In particular, $\eta_1,~\eta_2$ are units in $\mathcal{O}_\mathbb{K}/p\mathcal{O}_\mathbb{K}$, where $\mathbb{K}\coloneq \mathbb{Q}(\lambda_1,\lambda_2)$, and furthermore 
$(X_1,X_2)\equiv (1,1)\pmod p$.
Hence,
$$
(\lambda_1/\lambda_3)^{N_0}\equiv (\lambda_2/\lambda_3)^{N_0}\equiv 1\pmod p,
$$
so $\lambda_1^{N_0}\equiv \lambda_2^{N_0}\equiv \lambda_3^{N_0}\pmod p$.
This shows that if $N_0$ satisfies $T_{N_0}\equiv T_{N_0-1}\equiv 0\pmod p$, then 
$\lambda_1^{N_0}\equiv \lambda_2^{N_0}\equiv \lambda_3^{N_0}\pmod p$.
Conversely, if $\lambda_1^{N_0}\equiv \lambda_2^{N_0}\equiv \lambda_3^{N_0}\pmod p$ then  denoting by $\lambda^{N_0}$ the above common value modulo $p$, we get by the Binet formulas that 
\begin{align*}
T_{N_0} &= c_1\lambda_1^{N_0}+c_2\lambda_2^{N_0}+c_3\lambda_3^{N_0}\equiv \lambda^{N_0}(c_1+c_2+c_3) \equiv \lambda^{N_0}T_0\equiv 0\pmod p, \\
\shortintertext{and} 
T_{N_0-1} & = c_1\lambda_1^{N_0-1}+c_2\lambda_2^{N_0-1}+c_3\lambda_3^{N_0-1} \\
&\equiv \lambda^{N_0}(c_1\lambda_1^{-1}+c_2\lambda_2^{-1}+c_3\lambda_3^{-1}) \equiv \lambda^{N_0} T_{-1}\equiv 0\pmod p.
\end{align*}
This proves item (i).
For item (ii), note that 
$$
1=\left(\frac{\lambda_3}{\lambda_1}\right)^{N_0}\cdot \left(\frac{\lambda_3}{\lambda_2}\right)^{N_0}\cdot \left(\frac{\lambda_3}{\lambda_3}\right)^{N_0}=\left(\frac{\lambda_3^3}{\lambda_1\lambda_2\lambda_3}\right)^{N_0}
=\lambda_3^{3N_0},
$$
so $\lambda_3^{3N_0}\equiv 1\pmod p$ and the same is true for $\lambda_3$ replaced by $\lambda_1$ or $\lambda_2$.
This shows that $N\mid 3N_0$, so $k\in \{1,3\}$.
If $k>1$, then $\lambda^{N_0}\coloneq \lambda_p$ is different than $1$ and is an element of order $3$ modulo $p$.
This exists only if $3\mid p-1$, which proves (ii).
\end{proof}

We now move to an algebraic proof of Theorem \ref{mainr} for primes not equal to $2$ or $11$:
\begin{proof}
We first prove item \ref{p2}, and then item \ref{p1}.
For item \ref{p2}, assume that $a_0=a_1$.
Then
\begin{align*}
    \left(\frac{\lambda_1}{\lambda_3}\right)^{N_0} &= X_1=1+p^{a_0} \delta_1,\quad \text{where}\quad \delta_1\coloneq \frac{c_2/c_3}{\Delta\lambda_3^{N_0}} (b_0(\lambda_3/\lambda_2)-b_1) \\
    \shortintertext{and similarly,}
    \left(\frac{\lambda_2}{\lambda_3}\right)^{N_0} &= X_2=1+p^{a_0}\delta_2,\quad \text{where}\quad \delta_2\coloneq \frac{c_1/c_3}{\Delta\lambda_3^{N_0}} (b_1-b_0(\lambda_3/\lambda_1)).
\end{align*}
Note that 
\begin{equation} \label{eq:2}
    \delta_i=\frac{1}{p^{a_0}} \left(\left(\frac{\lambda_i}{\lambda_3}\right)^{N_0}-1\right) \quad \text{for}\quad i=1,2
\end{equation}
are conjugated algebraic numbers via the Galois automorphism of $\mathbb{K}$ which keeps $\lambda_3$ fixed (so also $c_3$ fixed) and swaps $\lambda_1$ and $\lambda_2$ (hence, also $c_1$ and $c_2$). 
We calculate
$$
N_{\mathbb{K}/\mathbb{Q}} (\delta_i)=\frac{1}{44} \left(b_0^6+4b_0^5b_1+11b_0^4b_1^2+12(b_0b_1)^3+11b_0^2b_1^4+4b_0b_1^5+b_1^6\right)
$$
for $i=1,2$, and we get a rational number whose denominator divides $44$.
By \eqref{eq:2} and the fact that $(\lambda_1/\lambda_3)^{N_0}-1$ is an algebraic integer, we get that in fact the norm of $\delta_1$ is a rational number whose denominator divides $p^{6a_0}$.
Since $p\ne 2,11$, we get that $\delta_1$ is an algebraic integer, therefore so is $\delta_2$. 
By induction on $\nu_p(m)$ using the binomial formula, we get that for all $m\ge 1$, 
\begin{align}
\label{eq:3}
\left(\frac{\lambda_2}{\lambda_3}\right)^{N_0m} &= 1+mp^{a_0}\delta_1 \pmod {p^{\nu_p(m)+a_0+1}\delta_1^2}), \ \text{and}\nonumber\\
\left(\frac{\lambda_2}{\lambda_3}\right)^{N_0m} & = 1+mp^{a_0}\delta_2 \pmod {p^{\nu_p(m)+a_0+1}\delta_2^2}).
\end{align}
Working backwards, we get
$$
c_1\lambda_1^{mN_0}+c_2\lambda_2^{N_0m}+c_3\lambda_3^{N_0m}\equiv mp^{a_0}b_0 \lambda_3^{N_0-1} \pmod {p^{\nu_p(m)+a_0+1}}.
$$
Thus,  have that 
$$
\frac{c_1\lambda_1^{mN_0}+c_2\lambda_2^{N_0m}+c_3\lambda_3^{N_0m}}{p^{\nu_p(m)+a_0}}\equiv \left(\frac{m}{p^{\nu_p(m)}}\right)b_0\lambda_0^{N_0-1} \pmod {p}.
$$
The number on the left is  a rational number while the number on the right is an algebraic integer coprime to $p$ since $p\nmid m/p^{\nu_p(m)}$ and $p\nmid b_0$.
This shows that the number on the left above is an algebraic integer coprime to $p$.
Thus, 
$$
\nu_p(T_{N_0m})=a_0+\nu_p(m)=\mu_p(T_{N_0})+\nu_p((N_0m)/N_0),
$$
which is what we wanted.
The same argument works for $\nu_p(T_{N_0m-1})$. 

The argument for item \ref{p1} is similar for the smaller of $a_0$, $a_1$.
Namely, let us assume for the sake of the argument that $a_0<a_1$.
We then get again that 
$$
\left(\frac{\lambda_1}{\lambda_3}\right)^{N_0}=1+p^{a_0}\delta_1\quad \text{and}\quad \left(\frac{\lambda_2}{\lambda_3}\right)^{N_0}=1+p^{a_0}\delta_2,
$$
except that now in the formulas for $\delta_1,~\delta_2$ we need to replace $b_1$ by $b_1p^{a_1-a_0}$.
Formulas \eqref{eq:3} still work.
Substituting formulas \eqref{eq:3}, we get again that 
$$
c_1\lambda_1^{mN_0}+c_2\lambda_2^{N_0m}+c_3\lambda_3^{N_0m}\equiv mp^{a_0}b_0 \lambda_3^{N_0-1} \pmod {p^{\nu_p(m)+a_0+1}},
$$
so that $\nu_p(T_{N_0m})=\nu_p(T_{N_0})+\nu_p(m)$.
However, when we substitute these in $T_{N_0m-1}$ we get that 
$$
c_1\lambda_1^{mN_0-1}+c_2\lambda_2^{N_0m-1}+c_3\lambda_3^{N_0m-1}\equiv mp^{a_1}b_1\lambda_3^{N_0-1} \pmod {p^{\nu_p(m)+a_0+1}}
$$
and since $a_1\ge a_0+1$, this shows that 
$$\nu_p(T_{N_0m-1})\ge a_0+1+\nu_p(m)>\nu_p(T_{N_0m}).
$$ 
If we assume instead that $a_0>a_1$, we can similarly show that $\nu_p(T_{N_0m}) > \nu_p(T_{N_0m-1})$ holds for all positive integers $m$. 
\end{proof}

\section{Tribonacci-Weifrich primes} \label{sec:weifrich}

Recall that by Equation \eqref{eq:binet}, 
$$
T_n=\sum_{i=1}^3 c_i\lambda_i^n\qquad \text{for}\quad n\in \mathbb{Z},
$$
where $c_i\coloneq \lambda_i/P'(\lambda_i)$ for $i=1,2,3$ are the roots of $P(X)=X^3-X^2-X-1$.
In this section we assume (although that is not needed), that $\lambda_1$ is real and $\lambda_2,~\lambda_3$ 
are complex conjugated.
Let $p\ne 2,11$.
Our previous arguments show that if $p^2\mid T_{N_0}$ and $p^2\mid T_{N_0-1}$, then if $\pi$ is any prime ideal of $\mathcal{O}_\mathbb{K}$
dividing $p$, then 
$$
\sum_{i=1}^3 c_i\lambda_i^{N_0-j}\equiv 0\pmod {\pi^2},\quad j=0,1,
$$
which in turn entails that $(\lambda_1/\lambda_2)^{N_0}\equiv 1\pmod {\pi^2}$. 
Furthermore, either $N_0$ or $3N_0$ is the least period of $\{T_n\}_{n}$ modulo $\pi$ so in particular $N_0\mid  N(\pi)-1$.
It follows that with $\varepsilon\coloneq \lambda_1/\lambda_2$, we have
\begin{equation} \label{eq:4}
    \varepsilon^{N(\pi)-1} -1\equiv 0\pmod {\pi^2}.
\end{equation}
Since $\varepsilon\in \mathcal{O}_\mathbb{K}$, where $\mathbb{K}\coloneq \mathbb{Q}(\alpha,\beta)=\mathbb{Q}(\varepsilon)$ is non-zero and not a root of unity, it follows that $\varepsilon$ is an \emph{admissible Wiefrich base} in the terminology of \cite{FM}.
Furthermore, a nonzero prime ideal $\pi$ for which relation \eqref{eq:4} is satisfied is called a \emph{base-$\varepsilon$ Wiefrich prime} by analogy with the classical Wiefrich primes which are primes $p$ such that $2^{p-1}\equiv 1\pmod {p^2}$.
For our situation, let us call a prime number $p$ such that there exists a prime ideal $\pi\in \mathcal{O}_\mathbb{K}$ with the property that relation \eqref{eq:4} holds a \emph{Tribonacci--Wiefrich prime}.
The ring $\mathcal{O}_\mathbb{K}$ has a very simple arithmetic structure.
Namely, it is monogenic, so $\mathcal{O}_\mathbb{K}=\mathbb{Z}[\eta]$, where $\eta\coloneq \sqrt[3]{\varepsilon}$ is an appropriate cubic root of $\varepsilon$ living inside $\mathbb{K}$ (see Proposition 5.1 in \cite{BLNOW}).
The minimal polynomial of $\eta$ is 
$$
X^6 +X^5 + 2X^4 + 3X^3 + 2X^2 +X + 1.
$$
Furthermore, $\mathcal{O}_\mathbb{K}$ has class number $1$.
Building on work of Silverman \cite{Sil}, Fellini and Murty show in Theorem 2.4 in \cite{FM} that a number field analogue of the Masser $abc$ conjecture (formulated by J. Browkin \cite{Brow} 20 years ago) for the field $\mathbb{K}$ implies that there are infinitely many non Tribonacci--Wiefrich primes.
So, we record this as a proposition:

\begin{proposition}
    The $abc$ conjecture for the number field $\mathbb{K}$ implies that there are infinitely many primes which are not Tribonacci--Wiefrich.
\end{proposition} 

In fact, the argument from \cite{FM} is quantitative giving that under the above assumptions the number of such primes $p\le x$ is at least $c_1\log x/\log\log x$ where $c_1>0$ is some positive constant.
This gives that under accepted conjectures, like the $abc$ conjecture, there should be infinitely many primes $p$ such that 
$$
\min\{\nu_p(T_{N_0}),\nu_p(T_{N_0-1})\}=1.
$$ 
We conjecture that there are infinitely many primes $p$ such that even 
$$\max\{\nu_p(T_{N_0}), \nu_p(T_{N_0-1})\}=1,
$$ 
but we have been unable to identify classical conjectures which may help to give an heuristic as to why this last conjecture should hold.
We leave such a task to the interested reader.  

\begin{proposition}
    Let $p\neq 3$ be a prime, then $p$ is Tribonacci-Weifrich if and only if $N(p)=N(p^2)$.
\end{proposition}
\begin{proof}
    We ignore $p=2,11$ since they are not Tribonacci-Weifrich and do not satisfy $N(p^2)=N(p)$.
    Now letting $p\nmid 66$, be Tribonacci-Weifrich, we have that $(\lambda_3/\lambda_i)^{N_0}\equiv 1\pmod{p^2}$ for $i=1,2,3$, we have that $\lambda_3^{3N_0}\equiv1\pmod{p^2}$.
    So $N(p^2)\mid 3N_0\mid 3N(p)$; since $3,p$ are relatively prime and $N(p^2)\mid pN(p)$, we must have that $N(p^2)=N(p)$.
    Conversely suppose $N(p^2)=N(p)$, thus $p^2\mid T_{3N_0},T_{3N_0-1}$; since at least one of the inequalities in Lemma \ref{lem5} holds, we can conclude that $\min\{\nu_p(T_{N_0}),\nu_p(T_{N_0-1})\}\ge 2$.  
\end{proof}

\subsection*{Computational Results} \label{ssec:compu}
A quick computation shows that $3$ is Tribonacci-Weifrich, but $N(9)\neq N(3)$.
Some code was written in SageMath to calculate $N_0(p)$ as well as $v_p(T_{N_0(p)})$ and $v_p(T_{N_0(p)-1})$, and it was found that for $ p < 10^7$, we have that $v_p(T_{N_0(p)}) = v_p(T_{N_0(p)-1}) = 1$ except for the cases shown in Table \ref{table:calcs}.
\begin{table}[h]
    \tbl{Table of values of $N_0(p)$, $v_p(T_{N_0(p)-1})$, and $v_p(T_{N_0(p)})$ for primes $p$ for which either of the latter two are greater than $1$.}
    {\begin{tabular}{cccc}
        \toprule
        $p$     & $N_0(p)$      & $v_p(T_{N_0(p)-1})$   & $v_p(T_{N_0(p)})$ \\ \midrule
        $2$     & $4$           & $1$                   & $2$ \\
        $3$     & $13$          & $2$                   & $2$ \\
        $5$     & $31$          & $2$                   & $1$ \\
        $7$     & $16$          & $2$                   & $1$ \\
        $97$    & $3169$        & $2$                   & $1$ \\
        $103$   & $17$          & $1$                   & $2$ \\
        $2621$  & $2620$        & $1$                   & $2$ \\
        $92377$ & $2844503376$  & $1$                   & $2$ \\ \bottomrule
    \end{tabular}}
    \label{table:calcs}
\end{table}

By \cite{MR2498633} we have that $3$ is the only Tribonacci-Weifrich prime less than $10^{11}$.

\section*{Acknowledgements}
% F. L. worked on this paper during a fellowship at STIAS in the second part of 2023. This author thanks STIAS for hospitality and support.
F. L. was supported in part by the 2024 ERC Synergy Grant ``DynAMiCs''.

\printbibliography

\end{document}